\documentclass[11pt]{amsart}
\usepackage{amsmath,mathrsfs,amsthm, amsrefs}
\usepackage{amssymb}
\usepackage{graphicx}
\usepackage{setspace}
\usepackage{braket}
\usepackage{mathtools}
\usepackage{cite}
\usepackage{xcolor}

\usepackage{soul}
\usepackage{fullpage}

\newcommand{\R}{\mathbb{R}}
\newcommand{\Z}{\mathbb{Z}}
\newcommand{\N}{\mathbb{N}}
\newcommand{\eps}{\varepsilon}
\newcommand{\al}{\alpha}
\newcommand{\be}{\beta}
\newcommand{\si}{\sigma}
\newcommand{\s}[1]{\mathscr{#1}}

\newtheorem*{conj*}{Conjecture}
\newtheorem{exmp}{Example}
\let\oldexmp\exmp
\renewcommand{\exmp}{\oldexmp\normalfont}
\newtheorem{lem}{Lemma}
\newtheorem{prop}{Proposition}

\newtheorem{thm}{Theorem}
\newtheorem{non-theorem}{Non-Theorem}

\theoremstyle{definition}
\newtheorem{defn}{Definition}

\theoremstyle{remark}
\newtheorem{remark}{Remark}

\title{Hausdorff Dimension of closure of cycles in $d$-maps on the circle}

\author{Nicholas Payne}
\address[N. ~P. ]{Northeastern University, Boston, MA 02115, USA}
\email{payne.n@northeastern.edu}

\author{Mrudul Thatte}
\address[M. ~T. ]{Columbia University, New York, NY 10027, USA}
\email{mrudul@math.columbia.edu}

\date{}

\begin{document}
\maketitle
\begin{abstract}
We study the dynamics of the map $x\mapsto dx (\text{mod }1)$ on the unit circle. We characterize the invariant finite subsets of this map which are called cycles and are graded by their degrees. By looking at the combinatorial properties of the base-d expansion of the elements in the cycles, we prove a conjecture of Curt McMullen that the Hausdorff dimension of the closure of degree-m cycles is equal to $\log m / \log d$. 
\end{abstract}

\section{Introduction}

Degree-$d$ maps on the unit disk have many interesting geometric, topological and analytic properties, which are closely related to hyperbolic geometry. The dynamics of these maps has an important application in classifying dynamical systems generated by polynomials in single complex variable ~\cite{goldberg1992fixed} and it provides useful information about Julia sets and Mandelbrot sets~\cite{keller2000invariant, douady1984etude}. To get an understanding of this family of maps, people studied the behaviour of these maps restricted to the boundary unit circle~\cite{goldberg1992fixed, goldberg1993fixed}. In the paper~\cite{mcmullen2010dynamics}, McMullen gave a complete description of simple (i.e. degree 1) cycles of these boundary maps, showing that those simple cycles are an analogue of simple closed geodesics on hyperbolic surfaces in the sense that the closure of union of simple cycles has Hausdorff Dimension 0 and the closure of union of simple closed geodesics has Hausdorff Dimension 1. 

Here we consider degree-$d$ holomorphic maps from the unit disk onto itself. We focus on one such map $z \mapsto z^d$. The restriction of this map to the boundary circle is equivalent to the map $x\mapsto dx \ \text{(mod 1)}$ on $\R/\Z.$ Combinatorial features of this map has been studied recently in~\cite{PZ, zakeri2021cyclic, tan2022counting}.
In this paper, we study cycles of higher degrees in this map, and in particular, calculate the Hausdorff dimension of their closure, which confirms a conjecture of McMullen. 

\begin{defn}[$d$-map]
 Let $d \in \N$. We define a \textbf{$d$-map} to be the map on the unit circle $S^1 = \R/\Z $ given by 
\begin{equation}
 x\mapsto dx \ \text{(mod 1)}, \forall x \in S^1.
\end{equation}
In other words, if the base-$d$ expansion of a point is $(0. b_1 b_2 b_3....)_d$, then $d$-map takes it to the point whose base-$d$ expansion is $(0. b_2 b_3b_4....)_d$.

We say that a $d$-map has degree $d$.
\end{defn}

\begin{defn}[Cycle]
Let $d$ be a positive integer greater than 1. A finite set $C\subset S^1$ is called a \textbf{cycle} for the $d$-map if and only if the $d$-map restricted to $C$ is a transitive permutation. In terms of the base-$d$ expansion, $C$ is given by
\begin{equation}
C = \{ (0. \overline{b_i b_{i + 1} .... b_n b_1 b_2 ....b_{i-1}})_d | 1 \le i \le n\},
\end{equation}
where $n$ is the size of the cycle and $b_1, b_2, ..., b_n$ are fixed digits in $\{0, 1, ..., d -1 \}$.
\end{defn}
Here are some examples of cycles. Consider the case $d = 2$. Here, $\{0\} = \{(0.\overline{0})_2\}$ is the only 1-element cycle and $\{\frac{1}{3},    \frac{2}{3}\} = \{(0.\overline{01})_2 ,(0.\overline{10})_2\}$ is the only 2-element cycle. For an integer $n > 2$, there are multiple $n$-element cycles. For example, there are two 3-element cycles: $\{\frac{1}{7}, \frac{2}{7}, \frac{4}{7} \}$ and $\{\frac{3}{7}, \frac{5}{7}, \frac{6}{7} \}$.

\newpage

Now we define an important invariant of a cycle, called degree.
\begin{defn}[Degree of a cycle]\label{def:deg}
Let $d$ be a positive integer greater than 1. Let $C$ be a cycle for $d$-map. The \textbf{degree} of $C$ is the smallest non-negative integer $m$ for which there exists a degree-$m$ map  
$f: S^1 \rightarrow S^1$ such that $f$ and the $d$-map agree on $C$. It is denoted by $deg(C)$.
\end{defn}
\begin{remark}
Note that $0 \le deg(C) \le d$ by definition.
\end{remark}
\begin{remark}
The cycles which contain only one element are fixed points of $d$-map and have degree 0. All other cycles have positive degree.
\end{remark}

Now, we give two examples.

\begin{exmp}
Consider $d = 2$ and $C = \{\frac{3}{7}, \frac{5}{7}, \frac{6}{7} \}$.  Under the 2-map, 
$$ \frac{3}{7} \mapsto \frac{6}{7}, \ \frac{5}{7} \mapsto \frac{3}{7} , \  \frac{6}{7} \mapsto \frac{5}{7}. $$
Here the cyclic order of the points in $C$ is preserved. As we jump on the points in the cycle from $ \frac{3}{7} \rightarrow \frac{5}{7} \rightarrow \frac{6}{7} $ (jump over 0) and back to $\frac{3}{7}$, we complete one full circle. As we jump on the images of these points, $\frac{6}{7} \rightarrow$ (jump over 0 to) $\frac{3}{7} \rightarrow \frac{5}{7}$ and back to $\frac{6}{7}$, again we complete one circle. This means that the degree of $C$ is 1. We can show this by explicitly constructing a suitable map $f: S^1 \to S^1$ as follows: 
\[ f(x) = \left\{ 
  \begin{array}{l l}
    x+\frac{3}{7} & \quad \text{if $ x \in \left[\frac{3}{7},  \frac{4}{7}\right] $}\\
    & \\
    3\left(x-\frac{4}{7}\right) & \quad \text{if $ x \in \left[\frac{4}{7},  \frac{5}{7}\right] $}\\
    & \\
    2x-1 & \quad \text{if $ x \in \left[\frac{5}{7},  \frac{6}{7}\right] $}\\
    & \\
   \frac{5}{7} + \frac{1}{4}\left(x-\frac{6}{7}\right)  & \quad \text{if $ x \in \left[\frac{6}{7},  \frac{10}{7}\right] $}\\
   \end{array} \right.\]

Observe that $f$ agrees with the 2-map, and it is of degree 1.  As there are no maps with smaller degree, $deg(C) = 1$.
\end{exmp}

\begin{exmp}
Consider $d = 3$ and $C = \{\frac{1}{5}, \frac{2}{5}, \frac{3}{5}, \frac{4}{5}\}$. Under the 3-map,
$$ \frac{1}{5} \mapsto \frac{3}{5}, \  \frac{2}{5} \mapsto \frac{1}{5}, \  \frac{3}{5} \mapsto \frac{4}{5}, \ \frac{4}{5} \mapsto  \frac{2}{5}.$$
So, as we jump on the points in the cycle from $ \frac{1}{5} \rightarrow \frac{2}{5} \rightarrow \frac{3}{5} \rightarrow \frac{4}{5}$ (jump over 0) and back to $\frac{1}{5}$, we complete one full circle. As we jump on the images of these points, $\frac{3}{5} \rightarrow$ (jump over 0 to) $\frac{1}{5} \rightarrow \frac{4}{5} \rightarrow $ (jump over 0 to) $\frac{2}{5} $ and back to $\frac{3}{5}$, we complete two circles. So, the degree of $C$ is 2. We can show this by explicitly constructing a suitable map $f: S^1 \to S^1$ as follows: 
\[ f(x) = \left\{ 
  \begin{array}{l l}
    \frac{3}{5}+ 4(x - \frac{1}{5}) & \quad \text{if $ x \in \left[\frac{1}{5},  \frac{3}{10}\right] $}\\
    & \\
   2(x - \frac{3}{10})  & \quad \text{if $ x \in \left[\frac{3}{10},  \frac{2}{5}\right] $}\\
       & \\
      3x - 1& \quad \text{if  $x \in \left[\frac{2}{5},  \frac{3}{5}\right]$}\\
      & \\                        
      \frac{4}{5}+ 2(x - \frac{3}{5})  & \quad \text{if  $x \in \left[\frac{3}{5},  \frac{7}{10}\right]$}\\
      & \\                  
      4(x - \frac{7}{10}) & \quad \text{if $ x \in \left[\frac{7}{10},  \frac{4}{5}\right] $}\\
      & \\      
      \frac{2}{5} + \frac{1}{2}\left(x - \frac{4}{5}\right)  & \quad \text{if $ x \in \left[\frac{4}{5},  \frac{6}{5}\right] $}\\
  \end{array} \right.\]

Observe that $f$ agrees with the 3-map, and it is of degree 2.  As there are no maps with smaller degree, $deg(C) = 2$.
\end{exmp}

\newpage

\begin{defn}[Closure of cycles]
Let $m, d \in \N$ with  $1 < d$ and $1 \le m \le d$. We define $E_{m, d}$ as the closure of the union of degree-$m$ cycles for $d$-map.
\begin{equation}
E_{m, d} = \overline{\{ c \in S^1| c \text{ is in a cycle of degree-$m$ for the $d$-map} \}}
\end{equation}
\end{defn}

Curtis McMullen discussed in ~\cite{mcmullen2010dynamics} the simple (i.e. degree 1) cycles for the $d$-map and and computed the Hausdorff Dimension of their closure $E_{1, d}$. Recall that Hausdorff Dimension is defined as follows.
\begin{defn}[Hausdorff Dimension]
The \textbf{Hausdorff Dimension} of a set $E$ is defined as
\begin{equation}
dim_H(E) = \inf\  \{\delta \ge 0 | \ \inf\ \{ \sum {r_i^{\delta}| E \subset \bigcup {B(x_i, r_i)}}\} = 0 \}.
\end{equation}
\end{defn}
\begin{thm}[McMullen]
Let $d$ be a positive integer greater than 1. Then $dim_H(E_{1, d}) =  0.$ 
\end{thm}

In this paper, we generalize the above results by showing the following.
\begin{thm}
Let $m, d \in \N$ with  $1 < d$ and $1 \le m \le d$. Then,
$$dim_H(E_{m, d}) = \frac {\log m}{\log d}.$$
\end{thm}

This paper is organized in the following way: 

In Section 2, we get a lower bound on the Hausdorff Dimension of $E_{m, d}$ by calculating the Hausdorff Dimension of a subset of  $E_{m, d}$. 

In Section 3, we prove the upper bound by imitating the proof of Theorem 1 in paper ~\cite{mcmullen2010dynamics} using combinatorial arguments, and conclude by proving Theorem 2.

\section{Lower Bound}

In this section, we define and study two useful invariants of a cycle called crossing number and Digit Portrait, which are directly related to the degree of the cycle. Then we use these properties to construct a subset of $E_{m, d}$ which has Hausdorff Dimension $\frac {\log m}{\log d}$.

\begin{defn}[Crossing]
Let $d$ be a positive integer greater than 1. Let $C =\{ c_1, c_2,..., c_n\}$ be a cycle for $d$-map such that  
$$0 < c_1 < c_2 <.... < c_n < 1$$ 
Let $c_{n + 1} = c_1$. For any $ 1 \le i \le n$, the pair $(c_i, c_{ i + 1})$ is called a \textbf{crossing generated by} $\mathbf{C}$ (or simply a \textbf{crossing}) iff 
\begin{equation}
0 < dc_{i + 1} \text{(mod 1)} < dc_i \text{(mod 1)} < 1
\end{equation}
The total number of such crossings is called the \textbf{crossing number} of C. 
\begin{equation}
\eta(C) = \text{\# (crossings generated by C)}
\end{equation}
\end{defn}

\begin{remark}
If $n = 1$, then there are no crossings. The crossing number is 0 which is the degree of 1-element cycles.
\end{remark}

\begin{remark}
As we jump on $S^1$ in the counterclockwise direction and trace the points of the cycle from $c_1 \rightarrow c_2 \rightarrow .... \rightarrow c_n$ and (jump over 0) back to $c_1$, we complete one full circle. When we trace the images of these points, $dc_1 \text{(mod 1)} \rightarrow dc_2 \text{(mod 1)} \rightarrow ...... \rightarrow dc_n \text{(mod 1)}$ and back to $dc_1 \text{(mod 1)}$, we may jump over 0 multiple times. Each time we jump over 0, we have a crossing. The crossing number of the cycle is equal to its degree. We provide a rigorous proof of this intuitive result below.
\end{remark}

\newpage

\begin{lem}\label{cross, deg} 
Let $d$ be a positive integer greater than 1. Let $C =\{ c_1, c_2,..., c_n\}$ be a cycle for $d$-map with $ 0 < c_1 < c_2 <.... < c_n < 1$. Then, the crossing number of $C$ is equal to its degree or $$\eta(C) = deg(C)$$
\end{lem}
\begin{proof}
The case $n = 1$ is obvious, so assume $n > 1$ and let $c_{n + k} = c_k$, for all $k \in \N.$

Let $1 \le i_1 < i_2 < ... < i_{\eta(C)} \le n $ be such that for all $1 \le t \le \eta(C)$, the pair $(c_{i_t}, c_{i_t + 1}) $ is a crossing generated by $C$. 

First, we prove $deg(C) \ge \eta(C)$.  Assume for the sake of argument that $\eta(C) > deg(C)$. We can divide $S^1$ or $[0, 1)$ in intervals $I_1, I_2, ... , I_{deg(C)}$ such that on each $I_r$, there exists a continuous non-decreasing map $g_r$ to $[0, 1)$ which agrees with the $d$-map on  $I_r \cap C$. \\

Note that there are exactly $\eta(C)$ sets of the type $\{c_{i_t}, c_{i_t + 1}, ... c_{i_{t + 1}}\}.$   By our assumption that $\eta(C) > deg(C),$ we can find a $t$ such that $\{c_{i_t}, c_{i_t + 1}, ... c_{i_{t + 1}}\} \subset I_r $ for some $r$. In other words, 
$$ 0 <  g_r(c_{i_t + 1}) = dc_{i_t + 1} \text{(mod 1)} < dc_{i_t} \text{(mod 1)} = g_r(c_{i_t})< 1,$$
a contradiction of the non-decreasing nature of $g_r$. Thus, $deg(C) \ge \eta(C)$.\\

Now we prove that $deg(C) \le \eta(C)$ by constructing an appropriate map $f : S^1 \rightarrow S^1$ of degree $\eta(C)$.  To start, we divide $S^1$ into $\eta(C)$ intervals of the type $\left[\frac{c_{i_t} + c_{i_t + 1}}{2}, \frac{c_{i_{t + 1}} + c_{{i_{t + 1}}  + 1}}{2} \right]$, each of which is mapped onto $S^1.$ We take the endpoints of each such interval to 0, and points of $C$ to their images under the $d$-map. In between, we make $f$ linear. We define
\[ f(x) = \left\{ 
  \begin{array}{l l}
     \frac{dc_{i_t + 1} \text{(mod 1)}}{\frac{c_{i_t + 1} - c_{i_t}}{2}} \ ( x - \frac{c_{i_t} + c_{i_t + 1}}{2}) & \quad \text{if $ x \in \left[\frac{c_{i_t} + c_{i_t + 1}}{2},  c_{i_t + 1}\right] $}\\
       & \\
      dc_i \text{(mod 1)} + {\frac{ dc_{i + 1} \text{(mod 1)} - dc_i \text{(mod 1)}}{c_{i + 1} - c_i}}\ ( x - c_i) & \quad \text{if  $x \in \left[c_i, c_{i + 1}\right] \subset \left[c_{i_t + 1}, c_{i_{t + 1}}\right]$}\\
      & \\                        
      dc_{i_{t + 1}} \text{(mod 1)} + {\frac{1 -  dc_{i_{t + 1}} \text{(mod 1)}}{\frac{c_{i_{t + 1} + 1} - c_{i_{t + 1}}}{2}}} \ (x - c_{i_{t + 1}}) & \quad \text{if } x \in \left[c_{i_{t + 1}}, \frac{c_{i_{t + 1}} + c_{{i_{t + 1}}  + 1}}{2} \right] 
  \end{array} \right.\]
where intervals are taken in counter-clockwise direction.

So, $f$ is a continuous function for all $t$, and the restriction of $f$ gives a bijection
\[\left[\frac{c_{i_t} + c_{i_t + 1}}{2}, \frac{c_{i_{t + 1}} + c_{{i_{t + 1}}  + 1}}{2}\right) \leftrightarrow S^1.\]

 Since $f$ has degree $\eta(C)$ and $f|_C$ agrees with the $d$-map, by Definition \ref{def:deg}, $deg(C) \le \eta(C).$  This, combined with above, completes the proof.

\end{proof}

Now that we have established the relation between the degree and the crossing number of a cycle, we need a tool to estimate the crossing number. We observe that the crossing number of a cycle is related to the order of points in the cycle, and hence the digits in the base-$d$ expansion of points of the cycle. We define an invariant of the cycle called Digit Portrait which characterizes these digits.

\begin{defn}[Digit Portrait]
Let $d$ be a positive integer greater than 1. Let $C$ be a cycle for $d$-map. The \textbf{Digit Portrait} of $C$ is the non-decreasing map 
$F: \{ 0, 1, 2, ..., (d - 1) \} \rightarrow  \{ 0, 1, 2, ..., |C| \} $ which satisfies 
$$F(j) = |\ C\ \cap \ [0, (j + 1)/d)\ | \ \ \forall \ 0 \le j \le (d - 1) $$ 
or
$$F(j) = \text{\# (elements of $C$ whose base-$d$ expansion starts with a digit less than $j + 1$)}.$$ 

Let $dig(C)$ be the number of distinct positive values taken by $F$. Note that if a digit $j$ is absent in the base-$d$ expansion of a point in $C$, then $F(j) = 0$ or $F(j) = F(j - 1)$. So, $dig(C)$  is also the number of distinct digits which appear in the base-$d$ expansion of any point in $C$. To estimate the crossing number of $C$, we need the second interpretation of $dig(C)$.
\end{defn}

\begin{exmp}
Consider $d = 4$ and the cycle $C = \{ (0 . \overline{0012})_4,   (0 . \overline{0120})_4, (0 . \overline{1200})_4,  (0 . \overline{2001})_4 \}$. Note that
$$ 0 = (0.0)_4 < (0 . \overline{0012})_4 <  (0 . \overline{0120})_4 < \frac{1}{4} = (0.1)_4 < (0 . \overline{1200})_4 < \frac{2}{4} = (0.2)_4 < (0 . \overline{2001})_3 < \frac{3}{4} = (0.3)_4$$

 The Digit Portrait of $C$ is the map $F: \{0, 1, 2, 3\} \rightarrow \{0, 1, 2, 3, 4\}$ given by:
$$F(0) = 2, F(1) = 3, F(2) = 4, F(3) = 4.$$

$F$ takes the values 2, 3 and 4. So, $dig(C)$ is 3. There are exactly 3 digits (0, 1 and 2) which appear in the base-4 expansion of the points in $C$. 
\end{exmp}

Now we establish the relation between $dig(C)$ and the crossing number of $C$.

\begin{lem}\label{cross, dig}
Let $d$ be a positive integer greater than 1, and $C =\{ c_1, c_2,..., c_n\}$ be a cycle for the $d$-map with $ 0 < c_1 < c_2 <.... < c_n < 1$ and $n > 1$. Then, the crossing number of $C$ is at most the number of distinct digits which appear in the base-$d$ expansion of a point in $C$.  In other words, $$\eta(C) \le dig(C).$$
\end{lem}
\begin{proof}
Let $1 \le i < n$. Let $c_i \in  [j_1/d, (j_1 + 1)/d)$ and $c_{i + 1} \in [j_2/d, (j_2 + 1)/d)$. In other words, the base-$d$ expansions of $c_i$ and $c_{i + 1}$ begin with the digits $j_1$ and $j_2,$ respectively. Note that $j_1 \le j_2 $ because $c_i < c_{i+1} $.

If $j_1 = j_2 = j$, then 
$$ 0 < dc_i \text{(mod 1)} =  dc_i - j  < dc_{i + 1} - j  = dc_{i + 1} \text{(mod 1)}< 1.$$

In this case, $ (c_i, c_{i + 1})$ cannot be a crossing. So, $(c_i, c_{i + 1})$ is a crossing only if $j_1 < j_2,$
 i.e., $c_i$ is the largest element of $ C \cap \ [j_1/d, (j_1 + 1)/d)$ and $c_{i + 1}$ is the smallest element of $ C\ \cap \ [j_2/d, (j_2 + 1)/d)$.

Thus, there are at most $(dig(C) - 1) \ i$'s for which  $1 \le i < n$ and  $ (c_i, c_{i + 1})$ is a crossing. For some cycles, $(c_n, c_1)$ is an additional crossing. So, there are at most $dig(C) \ i$'s for which  $1 \le i \le n$ and  $ (c_i, c_{i + 1})$ is a crossing.
\end{proof}

Together, Lemma \ref{cross, deg} and Lemma \ref{cross, dig} give a way of estimating the degree of a cycle by looking at the digits in the base-$d$ expansion of a point in the cycle. Now we use this to get a sufficient condition for a point to be in the closure of the cycles of fixed degree.

\begin{lem}\label{condition E_m,d}
Let $m, d \in \N$ with $1 \le  m \le  d$ and $1 < d$. Then, any point in $S^1$ whose base-$d$ expansion contains at most $m$ distinct digits lies in $E_{m, d}$.
\end{lem}
\begin{proof}
Let $\al \in S^1.$ Let $ \al = {(0.\al_1 \al_2 \al_3 .....)}_d$ such that $\forall \ r, \al_r \in \{b_1, b_2, ... ,b_m\} \subset \{0, 1, ..., (d - 1)\}.$ Here, $b_1, b_2, ..., b_m$ are fixed digits in base-$d$ such that $b_1 < b_2< ... < b_m.$

To prove that $\al$ lies in  $E_{m, d}$, we will show that for all $q \in \N, $ there exists a degree-$m$ cycle $C$ for  the $d$-map that intersects $d^{-q}$ neighborhood of $\al$. Any periodic point whose base-$d$ expansion contains exactly $m$ distinct digits is in a cycle of degree at most $m$. To get the maximum possible degree, we need maximum possible crossings. This can be achieved with the following construction:

Let $N \in \N$ such that for all $1 \le t \le m,$ $N$ is greater than the number of times $b_t$ appears in the first $q$ digits of $\al$. Let
$\langle b_t \rangle$ denote $  b_tb_tb_t...b_t \ \ (N \text{\ times} )$. Consider the following point:
\begin{equation}
c = {( 0 . \overline{\al_1\al_2\al_3...\al_q \langle b_m \rangle\langle b_1 \rangle\langle b_{m - 1} \rangle\langle b_1 \rangle....\langle b_2 \rangle\langle b_1 \rangle b_m})}_d
\end{equation}
Clearly, $c$ is in $d^{-q}$ neighborhood of $\al$. It is a point of cycle $C = \{c_1, c_2, ...c_{q + 2N(m - 1) + 1}\} $ for the $d$-map with 
 $0 < c_1 < c_2 < ... <  c_{q + 2N(m - 1) + 1}$. 

Now we need to prove that $deg(C) = m$.
For each $1 \le t \le m$, let $i_t$ be such that the largest element of $C$ whose base-$d$ expansion starts with the digit $b_t$ is $c_{i_t}$. 
Note that $c_{i_1}$ is at least ${(0 .b_1 b_m )}_d$. For $t > 1$, $\ c_{i_t}$ is at least ${(0 .\langle b_t \rangle )}_d$. 
For each $1 \le t < m, \  c_{i_t + 1}$ is the smallest element of $C$ whose base-$d$ expansion starts with the digit $b_{t + 1}$.
Note that the base-$d$ expansion of $ c_{i_t + 1}$ starts with $ 0.b_t \langle b_1 \rangle  $.

So, for all $1 \le t < m$,
$$ 0 < dc_{i_t + 1} \text{(mod 1)} < dc_{i_t} \text{(mod 1)} < 1,$$ 
i.e., $(c_{i_t}, c_{i_t + 1})$ is a crossing.

Note that $c_{q + 2N(m - 1) + 1} = c_{i_m}$ and $c_{i_m + 1} = c_1 = c_{i_1}$. So, 
$$ 0 < dc_{i_m + 1} \text{(mod 1)} < {(0. b_2)}_d < dc_{i_m} \text{(mod 1)} < 1,$$ 
i.e., $(c_{i_m}, c_{i_m + 1})$ is a crossing.

Thus, the crossing number of $C$ is at least $m$.
 $$\eta(C) \ge m$$
From Lemma \ref{cross, dig}, we know that the crossing number of $C$ is at most the number of distinct digits in the base-$d$ expansion of a point $C$, which in this case, is $m$. $$\eta(C) \le dig(C) = m$$

Therefore, $\deg(C) = \eta(C) = m$.\\
\end{proof}

The following result follows immediately.
\begin{prop} 
$E_{d, d} = S^1$. Thus, $dim_H(E_{d, d}) = 1 =  \frac{\log d}{\log d}.$
\end{prop}

\begin{proof}
Note that since base-$d$ has only $d$ digits, the set of all points of $S^1$ whose base-$d$ expansion contains at most $d$ distinct digits is $S^1$ itself. So, we have $S^1 \subset E_{d, d} \subset S^1$ or $E_{d, d} = S^1.$\\

\end{proof}

Now, we use Lemma \ref{condition E_m,d} to construct a subset of $E_{m, d}$.  Let $m$ and $d$ be positive integers with $d > 1$ and $1 \le m \le d$, and let $A_{m, d}$ be the set of points in $S^1$ whose base-$d$ expansion contains the digits from 0 to $(m-1)$ only. Clearly, $A_{m, d} \subset E_{m, d}$.

The structure of $A_{m, d}$ is similar to the structure of Cantor's set. Here, we start with $X_0 = [0, 1]$. Write $X_0$ in $d$ closed intervals of length $1/d$. Let $X_1$ be union of first $m$ of these intervals. 
$$X_1 = \bigcup _{i=0}^{m-1} {[i/d, {(i+1)}/d]} $$
$X_1$ is the union of $m$ intervals of length $1/d$. Divide each such interval in $d$ equal parts and take the first $m$ in $X_2.\ X_2$ is the union of $m^2$ intervals of length $1/{d^2}$.

Repeat the process. If $X_i$ is the union of $m^i$ intervals of length $d^{- i}$, then divide each such interval in  $d$ equal parts and take the first $m$ in $X_{i+1}$.
$$A_{m, d} = \bigcap _{i=0}^{\infty} {X_i}$$

Now we calculate the Hausdorff Dimension of $A_{m, d}$.

\newpage

\begin{lem}
$dim_H(A_{m, d}) = \frac {\log m}{\log d}$.

\end{lem}
\begin{proof}
First, we prove that the Hausdorff dimension of  $A_{m, d}$  is at most $\frac {\log m}{\log d}$. Let $\be > 0$. Note that for all $i,  A_{m, d} \subset X_i$. So, for each $i$, we have $m^i$ intervals of length $d^{-i}$ that form a covering of $ A_{m, d}$.
Letting $\be = \frac {\log m}{\log d} + \eps$ for some $\eps > 0$, 
$$ \lim_{i \to \infty} {m^i {(d^{-i})}^\be} = \lim_{i \to \infty} {m^i (d^{-i \eps})( (d^{\frac {\log m}{\log d}})^{-i}) } = \lim_{i \to \infty}{d^{-i \eps}} =  0  $$
This means that, for any $\be > \frac {\log m}{\log d}$, we can cover $A_{m, d}$ such that the summation of $\be$ powers of the lengths of the intervals in the cover is as small as we like. So, 
\[dim_H(A_{m, d}) \le \frac{\log m}{\log d}.\]

Now we need to show that the Hausdorff dimension of  $A_{m, d}$  is at least $\frac {\log m}{\log d}$. Note that we can consider  $A_{m, d}$ as a subset of $[0, m/d],$ a compact set. So, for any countable cover $(U_r)$ of $A_{m, d}$, we can find finitely many open sets $V_1, V_2, ..., V_p$ such that 
$$\bigcup _{r = 0}^{\infty} {U_r} \subset \bigcup _{r = 0}^p {V_r} \text{\quad and}$$
$$\sum _{r = 0}^p {|V_r|^\be} \le \sum _{r = 0}^{\infty} {|U_r|^\be}.$$

We next get a lower bound on $\sum _{r = 0}^p {|V_r|^\be}$. Let $b \in \N $ such that for all $r, d^{ - b} \le |V_r|$.  Additionally, for all $1 \le i \le b$, let $N_i$ be the number of $V_r$'s which satisfy $d^{-i} \le |V_r| < d^{- i +1}$. Observe that if $|V_r| < d^{- i + 1}$, then $V_r$ can intersect at most two intervals in $X_{i - 1}$. Hence,  $V_r$ can contain at most $2m^{b - i + 1}$ intervals in $X_b$, which has $m^b$ intervals. So,
$$\sum _{i = 0}^b {2m^{b - i + 1}N_i} \ge m^b $$
or
$$\sum  _{i = 0}^b {m^{-i} N_i} \ge \frac{1}{2m}$$
For  $\be = \frac {\log m}{\log d}$,
\[ \sum _{r = 0}^p {|V_r|^\be} \ge \sum _{i = 0}^b {{(d^{-i})}^\be N_i} = \sum  _{i = 0}^b {m^{-i} N_i} \ge \frac{1}{2m}.\]
In other words, the summation of $\frac{\log m}{\log d}$ powers of the lengths of the intervals which form a cover of $A_{m, d}$ has a positive lower bound. Thus, 
\[dim_H(A_{m, d}) \ge \frac {\log m}{\log d}.\]
\end{proof}

Since $A_{m, d} \subset E_{m, d},$ we immediately get a lower bound on the Hausdorff Dimension of $E_{m, d}:$

\begin{thm}
 Let $m, d \in \N$ with  $1 < d$ and $1 \le m \le d$. Then,
$$dim_H(E_{m, d}) \ge \frac{\log m}{\log d}$$
\qed
\end{thm}

\section{Upper Bound}

In this section, we first find an upper bound on the number of degree-$m$ cycles for $d$-map which have $n$ elements. We extend this result to precycles. Then, we find an appropriate covering of $E_{m, d}$ to prove that its Hausdorff Dimension is at most $\frac {\log m}{\log d}$.

\begin{defn}[Partition generated by a cycle]\label{def:partition}
Let $C =\{ c_1, c_2,..., c_n\}$ be a degree-$m$ cycle for $d$-map such that  $0 < c_1 < c_2 <.... < c_n < 1$. 
Let $\si$ be the map on $\{1, 2, ...., n\}$ which satisfies:
$$dc_r \text{(mod 1)} = c_{\si(r)} ,\forall \  1 \le r \le n$$  

Note that $\si$ is a permutation of $\{1, 2, ...., n\}$ because $C$ is a cycle.

Let $1 \le i_1 < i_2 < ... < i_m \le n $ be such that for all $1 \le t \le m, (c_{i_t}, c_{i_t + 1}) $ is a crossing generated by $C$. From the ordering of the elements of $C$ and the definition of crossing, we conclude that:
$$\si(i_t) > \si(i_t + 1) \text{ and } \si(i_t + 1) < \si(i_t + 2) <... < \si(i_{t + 1}), \ \forall \ 1 \le t < m$$ and
$$\si(i_m) > \si(i_m + 1) \text{ and } \si(i_m + 1) < ... < \si(n)  < \si(1) < ... < \si(i_1)$$

Now we construct a partition of $\{1, 2, .., n \}$ using the above property of $\si$.

\[ P_t = \left\{ 
  \begin{array}{l l}
    \{ \si(r) | i_t < r \le i_{t + 1} \} & \quad \text{if $1 \le t < m$}\\
    \{ \si(r) | i_m < r \le n \} \cup  \{ \si(r) | 1\le r \le i_1 \} & \quad \text{if $t = m$}
  \end{array} \right.\]

$ \{P_t | 1 \le t \le m \} $  is a partition of $\{1, 2, .., n \}$, called as the \textbf{partition generated by $C$} and it is denoted by $P(C)$.
\end{defn}
\begin{remark}
  Both $P(C)$ and the map $\si$ are useful counting $n$-element degree-$m$ cycles, as we show below.

\end{remark}
\begin{exmp}

Let us consider the instance where $d = 3$ and 
$$C = \{ (0. \overline{00102})_3, (0. \overline{01020})_3, (0. \overline{02001})_3, (0. \overline{10200})_3,  (0. \overline{20010})_3\}.$$ Note that here, $n = 5$. Under the 3-map,
\begin{align*}
c_1 &= (0. \overline{00102})_3 \mapsto (0. \overline{01020})_3 = c_2\\
c_2 &= (0. \overline{01020})_3 \mapsto (0. \overline{10200})_3 = c_4\\
c_3 &= (0. \overline{02001})_3 \mapsto (0. \overline{20010})_3 = c_5\\
c_4 &= (0. \overline{10200})_3 \mapsto (0. \overline{02001})_3 = c_3\\
c_5 &= (0. \overline{20010})_3 \mapsto (0. \overline{00102})_3 = c_1 
\end{align*}

So, $\si$ is the permutation of $\{1, 2, 3, 4, 5\}$ which takes 1, 2, 3, 4, 5 to 2, 4, 5, 3, 1 respectively. $(c_3, c_4)$ and  $(c_4, c_5)$ are the crossings generated by the cycle. So, $i_1 = 3$ and $i_2 = 4$, and the partition generated by $C$ is given by:
$$P(C) = \{P_1, P_2 \} \text{ where } P_1 =  \{\si(r)| 3 < r \le 4 \} , P_2 =  \{ \si(r)| 4 < r \le 5 \} \cup \{ \si(r)| 1 \le r \le 3\}$$
or
$$P(C) = \{ \{3\}, \{1, 2, 4, 5\}\}$$\\
\end{exmp}

Now, we will show that if some properties of a cycle such as degree, the partition it generates and its digit portrait are given, then we can construct the cycle (find its points). Later, we will use this result to count the number of cycles of fixed degree and cardinality.

\begin{lem}\label{cycle construct}
Given positive integers $d, m, n$,  $P = \{P_1, P_2,..., P_{m}\}$ which is a partition of $\{1, 2, ..., n\}$, positive integer $i_1 < |P_m| $ and a non-decreasing map  $F: \{ 0, 1, 2, ..., (d - 1) \} \rightarrow \{0, 1, 2, ..., n\} $ with $F(d - 1) = n$, there exists at most one cycle $C$ for $d$-map such that:
\begin{enumerate}
\item $C =\{ c_1, c_2,..., c_n\}$ with $0 < c_1 < c_2 <.... < c_n < 1$ 
\item $deg(C) = m$ 
\item $P(C) = P$ or $P$ is the partition generated by $C$
\item If, for all $1 < t \le m$, $i_t = i_{t - 1} + |P_t|$, then $(c_i, c_{ i + 1 })$ is a crossing for all $i \in \{i_1, i_2, ..., i_m\}$. 
\item $F$ is the Digit Portrait of $C$.
\end{enumerate}
\end{lem}

\begin{proof}
Suppose $C$ is a cycle for $d$-map which satisfies all the conditions above, and let $\si$ be the map on $ \{1, 2, ...., n\}$ given by:
\[\si(r) = \left\{ 
  \begin{array}{l l}
     {(r - i_t)}^{\text{th}}\text{ element of }P_{t}& \quad \text{if $i_t < r \le i_{t + 1}$ and $1 \le t < m$}\\
     {(r - i_m)}^{\text{th}}\text{ element of }P_{m}& \quad \text{if $i_m < r$}\\
     {(r + n - i_m)}^{\text{th}}\text{ element of }P_{m}& \quad \text{if $r < i_1$}\\
  \end{array} \right.\]
where, for all $1 \le t \le m$, the elements of $P_t$ are in increasing order. Note that $\si$ is uniquely determined by $m, n, P$ and $i_1$. Comparing this with the definition of partition generated by a cycle, we conclude that 
$$dc_r \text{(mod 1)} = c_{\si(r)} ,\forall \  1 \le r \le n$$ 
Let $b :  \{1, 2, ...., n\} \rightarrow \{0,1, 2, ...., (d - 1)\}$ given by:
\[b(r) = \left\{ 
  \begin{array}{l l}
     0 & \quad \text{if $r < F(0)$}\\
     j & \quad \text{if $F(j - 1) < r \le F(j)$ and $1 \le j \le (d - 1)$}
  \end{array} \right.\]\\
Note that $b$ is uniquely determined by $n, d$ and $k$. Comparing this with the definition of Digit Portrait, we conclude that
$$c_r \in [\ \frac{b(r)}{d}, \frac{b(r) + 1}{d}\ ), \ \forall \ 1 \le r \le n$$
or
$$b(r) = \text{the first digit in the base-}d\text{ expansion of }c_r, \text{for all} \ 1 \le r \le n$$\\
Also,
$$b(\si(r)) = \text{the first digit in the base-}d\text{ expansion of }c_{\si(r)} $$
$$\quad \quad = \text{ the second digit in the base-}d\text{ expansion of }c_r $$\\
Thus, we conclude that
$$c_r = {(0 .\overline{b(r)b(\si(r))b(\si^2(r)).....b(\si^{d- 1}(r))}\ )}_d$$ 
In other words, the cycle $C$ is uniquely determined by $\si$ and $b$. \\
\end{proof}

\begin{remark}
$(c_i, c_{i + 1})$ can be a crossing only if the first digits of the base-$d$ expansions of $c_i$ and $c_{i + 1}$ differ. So, cycle $C$ satisfying all conditions in the above lemma can exist only if 
$$\{ F(0), F(1), ...,  F(d - 1)\} \subset \{ i_1, i_2, ...i_m\} \cup \{n\}$$
We will use this in the proof of the following lemma.
\end{remark}

\newpage

\begin{lem}\label{count cycles}
Let $d, m, n \in \N$ with $1 \le m \le d$ and $n, d > 1$. Then the number of cycles $C$ for $d$-map satisfying $|C| = n $ and $deg(C) = m$ is at most $O(n^{d-m+1} m^{n-1}).$
\end{lem}

\begin{proof}
$d, m, n$ are given. Now we need  $P = \{P_1, P_2,..., P_{m}\}$ which is a partition of $\{1, 2, ..., n\}$, positive integer $i_1 < |P_m| $ and a non-decreasing map  $F: \{ 0, 1, 2, ..., (d - 1) \} \rightarrow \{0, 1, 2, ..., n\} $ with $F(d - 1) = n$ to fix a degree-$m$, $n$-element cycle.

Let $T$ be the set of ordered $(m + 1)$-tuples $(P_1, P_2, ..., P_m, i_1)$  such that
$P = \{P_1, P_2,..., P_{m}\}$ is a partition of $\{1, 2, ..., n\}$,
$i_1 \in \N$ and $ i_1 \le |P_m|$.
We want an upper bound on the size of T, denoted here as $\#(T)$.\\
\begin{align*}
\#(T) &= \sum _{\substack {a_1 + a_2 + ... + a_m = n \\ a_i \in \N}} {a_m\ {n \choose {a_1}} {{n - a_1} \choose {a_2}} ... {{{n - a_1 - a_2 - ... - a_{m-1}} \choose {a_m}}}}\\
&= \sum _{\substack {a_1 + a_2 + ... + a_m = n \\ a_i \in \N}} {a_m \cdot \frac{n!}{a_1! \ a_2! \ .... a_m!}}\\
&= \sum _{\substack {a_1 + a_2 + ... + a_m = n \\ a_i \in \N}} {\frac{n!}{a_1! \ a_2! \ .... a_{m - 1}! \ (a_m - 1)!}}\\
&\le \sum _{\substack {a_1 + a_2 + ... + a_{m-1} + (a_m - 1) = n - 1 \\ a_i \in \N \cup \{0\}}}{n \cdot  \frac{(n - 1)!}{a_1! \ a_2! \ .... a_{m - 1}! \ (a_m - 1)!}}\\
&= n \cdot m^{n - 1}
\end{align*}

After fixing an element of $T$, we need to fix a non-decreasing map $F$ from $ \{ 0, 1, 2, ..., (d - 1) \}$ to $\{0, 1, 2, ..., n\}$ with $F(d - 1) = n$ and $\{ F(0), F(1), ...,  F(d - 1)\} \subset \{ i_1, i_2, ...i_m\} \cup \{n\}$. Observe that there are $(n + 1)^{(d - m - 1)}$ or $(n + 1)^{(d - m)}$ choices for $F$, depending on if $n$ is in $\{ i_1, i_2, ...i_m\}$ or not.

Using Lemma \ref{cycle construct}, we conclude that the number of degree-$m$ $n$-element cycles for $d$-map is at most  $O(n^{d-m+1} m^{n-1})$.

\end{proof}
To conclude this section (and to show our desired result), we first introduce precycles.
\begin{defn}[Precycle]
Let $d$ be a positive integer greater than 1. A finite set $C_P \subset S^1$ is called a \textbf{precycle} for $d$-map, if and only if $C_P = \{c \cdot d^i \text{ (mod 1)}| i \in \N \cup \{0\}\}$ for some $c \in S^1$. In other words, a precycle is the forward orbit of a rational point in $S^1$. It can be written in terms of base-$d$ expansion of its points as
$$C_P = \{ (0.b_rb_{r + 1}...b_{n_1}\overline{b_1'b_2'....b_{n_2}'} )_d | 1 \le r \le n_1\} \cup \{  (0. \overline{b_k'b_{k+1}'...b_{n_2}b_1'....b_{k-1}'})_d | 1 \le k \le n_2 \}, $$
where $n_1$ is a fixed non-negative integer, $n_2$ is a fixed positive integer and all $b_r, b_k'$ are fixed digits in $\{0, 1, 2, ..., d - 1\}$. 
\end{defn}
\begin{remark}
Every precycle $C_P$ includes a cycle. If $n_1 = 0$, then $C_P$ itself is a cycle.
\end{remark}
\begin{exmp}
  Consider $d = 2$ and $C_P = \{ \frac{1}{3}, \frac{2}{3} , \frac{5}{6} \}$. $C_P$ can be written as
  $$ C_P = \{ \frac{5}{6} \cdot d^i \text{(mod 1)} | i \in \N \cup \{0\}\}$$
  or
  $$C_P = \{\frac{5}{6}\} \cup \{\frac{1}{3}, \frac{2}{3} \} = \{(0.1\overline{10})_2\} \cup \{(0.\overline{01})_2, (0.\overline{10})_2 \}$$
\end{exmp}

\begin{remark}
We can define degree, crossing, crossing number, digit portrait and partition for a precycle simply by replacing $C$ by $C_P$ in the respective definitions. \\
In the case where $C_P$ is a precycle but not a cycle, the map $\si$ in the definition of partition (Definition \ref{def:partition}) is not a permutation. Under $\si$, one element of  $\{1, 2, ..., n\}$ has no preimages, one has two preimages and all other elements have exactly one preimage. 
\end{remark}

\begin{lem}\label{count precycles}
Let $d, m, n \in \N$ with $1 \le m \le d$. Then the number of precycles $C_P$ for the $d$-map satisfying $|C_P| = n $ and $deg(C_P) = m$ is at most $O(n^{d-m+3} m^{n-1})$.
\end{lem}
\begin{proof}
  Similar to the argument used to prove Lemma \ref{count cycles}.  Note that the increase in the exponent of $n$ is due to the small change in the nature of the map $\si$ in the definition of partition (Definition \ref{def:partition}).\\
\end{proof}

Finally, we find a suitable covering of $E_{m, d}$ and get an upper bound on its Hausdorff Dimension, as desired.
\begin{thm}
Let $m, d \in \N$ with  $1 < d$ and $1 \le m \le d$. Then 
$$dim_H(E_{m, d}) \le \frac {\log m}{\log d}$$
\end{thm}

\begin{proof}
Fix a positive integer $n > 1$, and let 
$$\s C_P(n) = \{C_P | C_P \text{ is a precycle for $d$-map}, deg(C_P) \le m, |C_P| \le n\}.$$
From Lemma \ref{count precycles}, we have $|\s C_P|$ is at most $O(n^{d-m+4} m^n)$ (note the change in exponent due to the inequality).

Let $C$ be a degree-$m$ cycle for $d$-map, and $c \in C$. There exists a degree-$m$ map $f: S^1 \rightarrow S^1$ which agrees with $d$-map on $C$. Letting $f^n$ denote $f$ composed $n$ times, we note that $|(f^n)'(c)| = d^n$. See McMullen  ~\cite{mcmullen2010dynamics} for details. It follows there exists a point $x$ in $O(d^{-n})$ neighborhood of $c$ for which two of $x, f(x), f^2(x), ..., f^n(x)$ coincide. Generalizing, each point of any degree-$m$ cycle for the $d$-map lies in $O(d^{-n})$ neighborhood of an element in $\s C_P$. Thus, $E_{m, d}$  lies in a $O(d^{-n})$ neighborhood of $\s C_P$ which has at most $O(n^{d-m+4} m^{n})$ elements.\\

Let $\be = \frac{\log m}{\log d} + \eps$ for some $\eps > 0.$
\begin{align*}
\lim_{n \to \infty} {{(d^{-n})}^\be n^{d-m+4} m^{n}} &= \lim_{n \to \infty}{(d^{\frac {\log m}{\log d}})^{-n}d^{-n\eps} n^{d-m+4} m^{n}}\\
&=  \lim_{n \to \infty}{d^{-n\eps} n^{d-m+4}}\\
&= 0
\end{align*}
For any $\be > \frac {\log m}{\log d}$, we can cover $E_{m, d}$ such that the summation of $\be$ powers of the lengths of the intervals in the cover is as small as we like. So,  $dim_H(E_{m, d}) \le \frac{\log m}{\log d}$.\\

\end{proof}

We are now ready to prove Theorem 2.
\addtocounter{thm}{-3}
\begin{thm}
  Let $m, d \in \N$ with  $1 < d$ and $1 \le m \le d$. Then,
  $$dim_H(E_{m, d}) = \frac {\log m}{\log d}.$$
  \end{thm}
\begin{proof}
  Follows from Theorems 3 and 4.
\end{proof}

\ \\
\section*{Acknowledgment}
The authors would like to thank Prof. Curtis McMullen for suggesting the problems and giving valuable comments.  This project was undertaken by M.T. in SPUR (the Summer Program in Undergraduate Research) at MIT Mathematics Department in Summer 2014. M.T. would like to thank Prof. David Jerison and Prof. Pavel Etingof for their guidance and insightful discussions during the program. We also thank our mentor Xuwen Zhu for her help and support for completing this paper.

\nocite{*}
\bibliography{dMap}

\begin{bibdiv}
\begin{biblist}

\bib{birman1985geodesics}{article}{
      author={Birman, Joan~S},
      author={Series, Caroline},
       title={Geodesics with bounded intersection number on surfaces are
  sparsely distributed},
        date={1985},
     journal={Topology},
      volume={24},
      number={2},
       pages={217\ndash 225},
}

\bib{douady1984etude}{article}{
      author={Douady, Adrien},
      author={Hubbard, John~H},
      author={Lavaurs, P},
       title={Etude dynamique des polyn{\^o}mes complexes},
        date={1984},
}

\bib{goldberg1992fixed}{inproceedings}{
      author={Goldberg, Lisa~R},
       title={Fixed points of polynomial maps. i. rotation subsets of the
  circles},
organization={Soci{\'e}t{\'e} math{\'e}matique de France},
        date={1992},
   booktitle={Annales scientifiques de l'ecole normale sup{\'e}rieure},
      volume={25},
       pages={679\ndash 685},
}

\bib{goldberg1993fixed}{inproceedings}{
      author={Goldberg, Lisa~R},
      author={Milnor, John},
       title={Fixed points of polynomial maps. part ii. fixed point portraits},
organization={Soci{\'e}t{\'e} math{\'e}matique de France},
        date={1993},
   booktitle={Annales scientifiques de l'ecole normale sup{\'e}rieure},
      volume={26},
       pages={51\ndash 98},
}

\bib{keller2000invariant}{book}{
      author={Keller, Karsten},
       title={Invariant factors, julia equivalences and the (abstract)
  mandelbrot set},
   publisher={Springer},
        date={2000},
}

\bib{mcmullen2010dynamics}{article}{
      author={McMullen, C},
       title={Dynamics on the unit disk: short geodesics and simple cycles},
        date={2010},
     journal={Comment. Math. Helv},
      volume={85},
       pages={723\ndash 749},
}

\bib{mcmullen1998hausdorff}{article}{
      author={McMullen, Curtis~T},
       title={Hausdorff dimension and conformal dynamics, iii: Computation of
  dimension},
        date={1998},
     journal={American journal of mathematics},
       pages={691\ndash 721},
}

\bib{mcmullen1999hausdorff}{article}{
      author={McMullen, Curtis~T},
       title={Hausdorff dimension and conformal dynamics i: Strong convergence
  of kleinian groups},
        date={1999},
}

\bib{mcmullen2000hausdorff}{article}{
      author={McMullen, Curtis~T},
       title={Hausdorff dimension and conformal dynamics ii: Geometrically
  finite rational maps},
        date={2000},
     journal={Commentarii Mathematici Helvetici},
      volume={75},
      number={4},
       pages={535\ndash 593},
}

\bib{PZ}{article}{
      author={Petersen, Carsten~L},
      author={Zakeri, Saeed},
       title={On combinatorial types of periodic orbits of the map $x\mapsto kx
  (mod \mathbb{Z})$},
        date={2020},
     journal={Advances in Mathematics},
      volume={361},
       pages={106953},
}

\bib{tan2022counting}{article}{
      author={Tan, Yee~Ern},
       title={Counting rotational subsets of the circle $\mathbb{R} /
  \mathbb{Z}$ under the angle multiplying map $ t\mapsto dt$},
        date={2022},
     journal={arXiv preprint arXiv:2207.03594},
}

\bib{zakeri2021cyclic}{article}{
      author={Zakeri, Saeed},
       title={Cyclic permutations: Degrees and combinatorial types},
        date={2021},
     journal={Journal of Combinatorial Theory, Series A},
      volume={184},
       pages={105518},
}

\end{biblist}
\end{bibdiv}
\bibliographystyle{amsplain}

\end{document}